\documentclass[11pt]{article}
\usepackage{amssymb,latexsym,amsmath,amsthm}
\usepackage{mathrsfs}
\renewcommand{\epsilon}{\varepsilon}
\newcommand{\RR}{\mathbb{R}}
\newcommand{\NN}{\mathbb{N}}
\newcommand{\sC}{\mathscr{C}}
\newcommand{\sG}{\mathscr{G}}
\newcommand{\sL}{\mathscr{L}}
\newcommand{\sR}{\mathscr{R}}
\newcommand{\sW}{\mathscr{W}}
\newcommand{\Xp}{{\mathfrak{X}^p}}
\newcommand{\rd}{\mathrm{d}}

\newtheorem{theorem}{Theorem}
\newtheorem{lemma}{Lemma}

\begin{document}

\begin{center}
{\bf\Large Compactness and stability for planar vortex-pairs with
prescribed impulse}\\
\bigskip
{\bf\large G. R. Burton}\\
{\em Department of Mathematical Sciences, University of Bath, Bath BA2 7AY, U.K.}\\
\end{center}

\bigskip

\noindent
{\sc Abstract}.-
Concentration-compactness is used to prove compactness of maximising sequences
for a variational problem governing symmetric steady vortex-pairs in a
uniform planar ideal fluid flow, where the kinetic energy is to be maximised
and the constraint set comprises the set of all equimeasurable rearrangements
of a given function (representing vorticity) that have a prescribed impulse
(linear momentum).
A form of orbital stability is deduced.

\bigskip

\noindent
{\sc R\'{e}sum\'{e}}.-
On utilise la methode de compacit\'{e} par concentration pour d\'{e}monstrer
la compacit\'{e} des s\'{e}quences de maximisation pour un probl\`{e}me
variationnel d\'{e}scribant des pairs-tourbillons symm\'{e}triques dans
un \'{e}coulement fluide planaire uniforme, o\`{u} on maximise l'en\'{e}rgie
cin\'{e}tique dans l'ensemble des r\'{e}arrangements mesurables
d'un fonction (repr\'{e}sentant le tourbillon) aux impulsion prescrite
(moment lin\'{e}aire). On en d\'{e}duit une forme de stabilit\'{e} orbital.

\bigskip


\bigskip

\section{Introduction.}\label{S1}
In this paper we prove compactness (up to translation)
of all maximising sequences for a constrained variational problem whose
maximisers represent steady axisymmetric vortex-pairs in a planar flow
of an ideal fluid of unit density that approaches a uniform flow at infinity.
The quantity being maximised is the kinetic energy $E(\zeta)$ due to
the vorticity $\zeta$ and
the constraints are that the vorticity in the upper half-plane should be a
rearrangement of a given compactly supported non-negative function and that
the impulse (linear momentum parallel to the axis) $I(\zeta)$ should take a
prescribed value.
This formulation of steady vortex pairs,
which is derived from ideas of
Arnol$^\prime$d \cite{VIA:COND} and
Benjamin \cite{TBB},
has the notable feature that all the prescribed data are
conserved quantities in the corresponding dynamical problem.

We apply the compactness theorem to orbital stability:
we prove that symmetric flows with non-negative upper half-plane vorticity
starting close to a maximiser remain close to the set of maximisers
for all time, relative to a norm defined by
\[
\| \zeta  \|_\Xp := |I(\zeta)| + \| \zeta \|_1 + \| \zeta \|_p ,
\]
where $p>2$.
It must be emphasised that this result does not preclude arbitrarily small
perturbations that result in two-signed initial vorticity leading to flows
that become distant from the set of maximisers.
This stability result is a counterpart for one proved in \cite{BLL}
where kinetic energy penalised by a given constant multiple of impulse was
maximised relative to a set of rearrangements.

In the present paper the same norm $\| \cdot \|_\Xp$ is used to
measure deviations in both the initial state and the evolved state,
whereas in \cite{BLL} deviations in the evolved state were
measured in a weaker norm than deviations in the initial state.
The precise formulations of the results are given
in Subsection \ref{results}; the sequences considered in the compactness
theorem are in fact slightly more general than maximising sequences.

While it is clear that all maximisers of the variational problem of \cite{BLL}
are maximisers of the problem of the present paper (for appropriate fixed
values of the impulse) it is unclear whether the converse holds.
In particular, we do not know generally whether the maximisers for either
formulation are unique (up to axial translation), or even whether all
maximisers for the penalised energy formulation in \cite{BLL}
have the same impulse.

Stability of planar ideal fluid flows has been the subject of many
investigations.
The outcome can depend on which norm is used in the definition of stability,
especially the order of derivatives included, if any, a point that is
emphasised in the expository articles by
Friedlander and Shnirelman \cite{FrShn} and
Friedlander and Yudovich \cite{FrYud}.

\subsection{Formulation.}
\label{formul}

We take the flow to be symmetric in the $x_1$-axis of $\RR^2$, so we work
in the upper half-plane $\Pi=\{ (x_1,x_2) \in \RR^2) \mid x_2>0 \}$
and let  $G$ be the Green's function of $-\Delta$ in $\Pi$, that is
\begin{equation}
G(x,y) := \frac{1}{2\pi} \log \left (\frac{|x-\overline{y}|}{|x-y|} \right)
= \frac{1}{4\pi} \log \left( 1+ \frac{4x_2 y_2}{|x-y|^2} \right)
\label{eqx6}
\end{equation}
where $\overline{y} = (y_1,-y_2)$ denotes the reflection in the $x_1$-axis
of $y=(y_1,y_2)$.
An operator $\sG$ is defined by
\[
\sG\zeta(x)= \int_\Pi G(x,y) \zeta(y) \rd y,
\]
when this integral converges (for which it is sufficient that
$\zeta \in L^1(\Pi) \cap L^p(\Pi)$ for some $p>1$).
For a flow of an ideal (incompressible, inviscid) fluid of unit density
in $\Pi$ approaching a uniform stream with velocity $(\lambda,0)$ at infinity,
parallel to $e_1=(1,0)$ on the axis and having signed scalar vorticity
$\zeta(x)$, the stream function is given by $\sG\zeta(x) - \lambda x_2$ and
the kinetic energy $E(\zeta)$ and impulse $I(\zeta)$ are given by
\begin{eqnarray*}
E(\zeta) &=& \frac{1}{2}\int_\Pi \zeta(x) \sG\zeta(x) \rd x
= \frac{1}{2} \int_\Pi G(x,y) \zeta(x) \zeta(y) \rd x \rd y ,\\
I(\zeta) &=&  \int_\Pi \zeta(y)y_2 \rd y .
\end{eqnarray*}
Note that $E(\zeta)>0$ if $\zeta \neq 0$.

Consider a fixed compactly supported non-negative $\zeta_0 \in L^p(\Pi)$,
for some $2<p<\infty$ and let $\sR(\zeta_0)$ denote the set of all
rearrangements of $\zeta_0$ on $\Pi$, defined in Sect. \ref{rearr} below.
Given a number $i_0>0$, any maximiser of $E(\zeta)$ subject to the constraints
$\zeta \in \sR(\zeta_0)$ and $I(\zeta)=i_0$ satisfies an equation
\begin{equation}
\label{eq5}
-\Delta \psi(x) = \varphi (\psi(x) - \lambda x_2) \quad \mbox{ in } \Pi
\end{equation}
where $\psi := K\zeta$ and $\varphi$ is an ({\em a priori} unknown) increasing
function; if $\psi$ satisfies \eqref{eq5} then $\psi(x_1,x_2)- \lambda x_2$
is the 
stream function of a steady ideal fluid flow.
Viewed in a frame fixed relative to the fluid at infinity,
$\overline{\psi}(x_1,x_2,t):=\psi(x_1 + \lambda t,x_2)$, where $\psi$ satisfies
\eqref{eq5}, is the stream function for a vortex of constant form moving
with velocity $(\lambda,0)$ in an otherwise irrotational fluid that is
stationary at infinity.

Our purpose here is to prove that given $\zeta_0$, for all sufficiently large
$i_0$ the maximisers are orbitally stable to non-negative perturbations of
vorticity in $\| \, \|_\Xp$ in the sense that, if $\zeta$ is a maximiser
and $\omega(t)$ is a flow of a planar ideal fluid whose initial vorticity
$\omega(0)$ is non-negative and is close to $\zeta$ in $\| \,\|_\Xp$ then
$\omega(t)$ remains close in $\| \,\|_\Xp$ to the set of maximisers.

We work with a relaxed formulation of the variational problem,
having a convex constraint set, and refer to its solutions as
{\em relaxed maximisers}; we show these must always exist.
It transpires (see Lemma \ref{lmx7}) that, given $\zeta_0$,
for all sufficiently large $i_0$ the relaxed maximisers 
are in fact solutions to the original unrelaxed problem;
it is under these circumstances that we can prove orbital stability.

\subsection{Background.}

The notion that extrema of kinetic energy relative to a set of equimeasurable
vorticities, or ``equivortical surface'', should provide stable steady flows,
was introduced by Arnol$^\prime$d \cite{VIA:COND}.
Benjamin \cite{TBB} adapted Arnol$^\prime$d's ideas to the context of steady
axisymmetric vortex-rings in a uniform flow in the whole of $\RR^3$ and
proposed a programme for proving the existence of energy maximisers subject
to the additional constraint of fixed impulse and for proving their
stability from this.
The problem considered in the present paper is the two-dimensional analogue
of Benjamin's problem and this work represents a contribution to its stability
theory, existence of solutions having been studied previously in \cite{GRB:VP}.

Arnol$^\prime$d's stability method proceeds by constructing a Lyapunov
functional, derived from the vorticity-stream function relationship,
concerning which some degree of detail must therefore be known.
However, Benjamin's approach to steady vortices has the novel feature that the
functional relationship between the stream function and the vorticity, which
is the classical condition for a planar flow to be steady, is not specified
{\em a priori} but is determined {\em a posteriori}, by contrast with much
other work on existence of solutions to semilinear elliptic equations.
Our result will therefore not be derived by Arnol$^\prime$d's method,
but rather by the approach to stability envisaged by Benjamin \cite{TBB}.

Burton, Lopes and Lopes \cite{BLL} proved a stability theorem for a related
formulation where the energy is penalised by subtracting a fixed multiple of
the impulse and then maximised on a set of rearrangements without any
constraint on the impulse, which is the planar analogue of an alternative
formulation of vortex rings that had also been proposed by Benjamin \cite{TBB}.
In \cite{BLL} it was shown that when the initial vorticity is non-negative
and close to a maximiser in terms of $I$ and $\| \, \|_2$, subject to a
bound on the area of the vortex-core, the vorticity of the evolving flow
remains close in $\| \, \|_2$ to the set of maximisers.
We suggest that the form of stability proved in the present paper is more
elegant, in using the same norm to measure the perturbations of both the
initial state and the evolved state.

\subsection{Methodology.}

The proof of stability in \cite{BLL} proceeded by constructing, given a flow
with vorticity $\omega(t)$ for which $\omega(0)$ is close to a maximiser
$\widehat{\zeta}$, a ``follower'' $\zeta(t)$ in the constraint set by
advecting $\widehat{\zeta}$ using the transport equation of the velocity
field associated with $\omega(t)$; the distance between $\omega(t)$ and
$\zeta(t)$ is constant in time.
The result was then deduced from a compactness theorem for the maximising
sequences of the variational problem, whose proof used the
Concentration-Compactness principle expounded by P.-L. Lions \cite{PLL:CC1}.

In the present context, this approach gives rise to two difficulties.
Firstly, the transport equation need not conserve impulse so its solutions
need not remain in the constraint set; this frustrates the construction
of a follower.
Secondly, the absence of the penalty term on the energy increases
the difficulty of bounding the supports of various vorticities arising in
the ``Dichotomy'' case of Concentration-Compactness.
We therefore prove directly the compactness of maximising sequences that
approach the constraint set but need not be contained within it and, when
addressing Dichotomy, we improve the compactness by solving subsidiary
constrained variational problems, which may have relaxed solutions only.

\section{Preliminaries and Statements of Results.}
\label{prelim}

In this section we review the properties of the set $\sR(\zeta_0)$ of
rearrangements of a function $\zeta_0$ and we define the constraint set
$\sW(\zeta_0,\leq i_0)$ for the relaxed variational problem.
Then we derive some estimates for the stream function in terms of the
vorticity and establish weak continuity properties of the energy, which
will be needed for the study of relaxed problem in Section \ref{exist}.
Throughout we use the notation $D(x,R)=\{ y \in \RR^2 \mid |y-x| < r \}$
and $D_\Pi(x,R) = D(x,R) \cap \Pi$.

\subsection{Rearrangements and Steiner symmetrisation.}
\label{rearr}

Suppose that $f$ and $g$ are non-negative functions $p$-integrable (for some
$1 \leq p <\infty$) on sets of infinite measure in Euclidean spaces (possibly
of different dimensions).
We say $f$ is a {\em rearrangement} of $g$ if the Lebesgue measures of
corresponding super-level sets of $f$ and $g$ are equal, that is, if
\[
\forall \alpha>0 \quad |\{ x  \mid f(x)>\alpha \}|
= |\{ y  \mid g(y)>\alpha \}|
\]
where $| \, |$ denotes Lebesgue measure of appropriate dimension.
We write $f \preceq g$ if
\begin{equation}
\forall \alpha>0 \; \int_U (f(x)-\alpha)_+\rd x \leq
\int_V (g(y) -\alpha)_+ \rd y  .
\label{eqx8}
\end{equation}
Thus $\preceq$ is transitive, while $f$ is a rearrangement of $g$ if and only
if both $f \preceq g$ and $g \preceq f$ hold, since the measures of the
super-level sets of $f$ and $g$ can be recovered by right-differentiation
of the integrals in \eqref{eqx8} with respect to $\alpha$.
Typically we are interested in functions defined on half-planes, planar strips
or unbounded intervals in $\RR$.
Any set of finite positive Lebesgue measure in a Euclidean space
(of any dimension) is measure-theoretically isomorphic to an interval of
equal linear measure in $\RR$,
indeed a measure-preserving bimeasurable bijection to the interval less
a countable set exists, see \cite[Chapter 15]{HLR}, and consequently any
set of infinite measure is isomorphic to a half-line and to $\RR$.
The isomorphisms allow a measurable function on one set to be
lifted to a rearrangement on any other set of equal measure.
We can therefore be somewhat cavalier about domains in what follows.

The essentially unique {\em decreasing rearrangement} $f^\Delta$ of $f$ can be
defined on $(0,\infty)$ and we extend $f^\Delta$ to vanish on $(-\infty,0)$.
We define the {\em increasing rearrangement} of $f$ by
$f^\nabla(s)=f^\Delta(-s)$ for real $s$.
Clearly $f$ is a rearrangement of $g$ if and only if $f^\Delta=g^\Delta$.
The inequality
\begin{equation}
\| f^\Delta - g^\Delta \|_p \leq \| f - g \|_p
\label{eq7}
\end{equation}
holds if $f,g \in L^p(\RR^n)$ for some $1 \leq p \leq \infty$, see for example
\cite[Theorem 3.5]{LILO}, thus decreasing rearrangement is continuous in $L^p$.
One can think of $f^\Delta$ as capturing the statistics of the values of $f$
without capturing their spatial distribution.

Given non-negative $f \in L^p(\Pi)$, for some $1 \leq p \leq \infty$, and
$\beta \in \RR$, the {\em Steiner symmetrisation of $f$ in the line
$x_2=\beta$} is the (essentially unique) non-negative function $f^s$ on $\Pi$
such that, for almost every $x_2>0$, the function $f^s(\cdot,x_2)$ is
symmetric decreasing about $\beta$ and $f^s(\cdot,x_2)$ is a rearrangement
of $f(\cdot,x_2)$.

Consider a non-negative $p$-integrable function $\zeta_0$ (some $2<p<\infty$)
defined on a subset of infinite Lebesgue measure in a Euclidean space and
let $U$ be another subset of infinite measure in a Euclidean space.
We denote by $\sR_U(\zeta_0)$ the set of all rearrangements of $\zeta_0$ on $U$
and define the set $\sW_U(\zeta_0)$, which contains $\sR_U(\zeta_0)$, by
\[
\sW_U(\zeta_0) = \left\{ 0 \leq \zeta \in \left. L^1(U) \,
\right|  \, \zeta \preceq \zeta_0  \right\}  ,
\]
writing in particular $\sW(\zeta_0)=\sW_\Pi(\zeta_0)$ and
$\sR(\zeta_0)=\sR_\Pi(\zeta_0)$.
Then, from the definition, $\sW(\zeta_0)$ is convex and
Douglas \cite{RJD:UD} proved that $\sW(\zeta_0)$ is the closure of
$\sR(\zeta_0)$ in the weak topology of $L^p(\Pi)$. Let
\begin{eqnarray*}
\sW(\zeta_0,\leq i_0) &=&
\{\zeta \in \sW(\zeta_0) \mid I(\zeta) \leq i_0\} \\
\sW(\zeta_0,i_0) &=&
\{\zeta \in \sW(\zeta_0) \mid I(\zeta) = i_0\} \\
\sW(\zeta_0,< \infty) &=&
\{\zeta \in \sW(\zeta_0) \mid I(\zeta) < \infty \} .
\end{eqnarray*}
Then $\sW(\zeta_0,\leq i_0)$ is convex and strongly closed in $L^p$, and
therefore weakly closed, whereas $\sW(\zeta_0,i_0)$ is not strongly closed.
Further if $\zeta \in \sW(\zeta_0)$ then $\|\zeta\|_1 \leq \|\zeta_0\|_1$
and $\|\zeta\|_p \leq \|\zeta_0\|_p$.

We write
\begin{eqnarray*}
\sR^+(\zeta_0) &=& \{ \zeta 1_A \mid \zeta \in \sR(\zeta_0), \; A \subset \Pi
\mbox{ measurable } \}, \\
\sR\sC(\zeta_0) &=& \{ 0 \leq \zeta \in L^1(\Pi) \mid
\zeta^\Delta = \zeta_0^\Delta 1_{(0,\ell)},
\mbox{ some } 0 \leq \ell \leq \infty \}.
\end{eqnarray*}
The elements of $\sR\sC(\zeta_0)$ were called {\em curtailments of
rearrangements} of $\zeta_0$ by Douglas \cite{RJD:UD}, who showed that
$\sR\sC(\zeta_0)$ is the set of extreme points of $\sW(\zeta_0)$.
The ideas described above form counterparts, for domains of infinite measure,
of some results of Ryff \cite{Ryff:Major} for bounded intervals
(and hence, by isomorphism, for more general domains of finite measure).

We will consider the relaxed problem of maximising $E$ relative to
$\sW(\zeta_0,\leq i_0)$ and let $M_0$ denote the supremum of this problem.
It is easy to see from the second form of the Green's function $G$
in \eqref{eqx6}
that translating any nontrivial $\zeta \geq 0$ in the positive $x_2$-direction
strictly increases $E$, so any maximiser of $E$ relative to
$\sW(\zeta_0,\leq i_0)$ must belong to $\sW(\zeta_0,i_0)$.
The strict convexity of $E$ shows that maximisers of $E$ relative to
$\sW(\zeta_0,i_0)$ are extreme points of this set; together with
a result of Douglas \cite{RJD:UD} this will be used to show that all
maximisers belong to the set $\sR\sC(\zeta_0)$. 
This is not quite straightforward because $I$ is an unbounded functional 
so Douglas's theorem does not apply directly to $\sW(\zeta_0,i_0)$.

\subsection{Results.}\label{results}
We prove the following compactness theorem, which is more general than
compactness of maximising sequences:

\begin{theorem}[{\bf Compactness}]
Let $i_0 >0$ and $\zeta_0 \in L^p(\Pi)$, $p>2$, be given, such that
$\zeta_0$ is nontrivial and non-negative with compact support.
Let $\Sigma_0$ be the set of maximisers of $E$ relative to
$\sW(\zeta_0,i_0)$, let $M_0$ be the value of $E$ on $\Sigma_0$ and
suppose $\emptyset \neq \Sigma_0 \subset \sR(\zeta_0)$.
Suppose $(\zeta_n)$ is a non-negative sequence in $L^p \cap L^1 (\Pi)$ such
that $\zeta_n^\Delta \to \zeta_0^\Delta$ in both $L^p$ and $L^1$,
$I(\zeta_n) \to i_0$ and $E(\zeta_n) \to M_0$.
Then some subsequence of $(\zeta_n)$ converges, after $x_1$-translation,
in both $L^p$ and $L^1$ to an element of $\Sigma_0$.
\label{thm1}
\end{theorem}

\noindent
{\bf Definition.}
By an {\em $L^p$-regular solution} of the vorticity equation we mean
$\omega \in L^\infty_\mathrm{loc}([0,\infty)), L^1(\Pi)) \cap
L^\infty_\mathrm{loc}([0,\infty)), L^p(\Pi))$ satisfying, in the sense of
distributions, 
\begin{equation}
\label{eq8}
\begin{cases}
\partial_t \omega + \mathrm{div} (\omega u) = 0,\\
u = \nabla^\perp \sG \omega, \quad (t,x) \in [0,\infty) \times \Pi ,
\end{cases}
\end{equation}
such that $E(\omega(t,\cdot))$ and $I(t,\cdot)$ are constant.
(The operators $\mathrm{div}$, $\nabla^\perp$ and $\mathscr{G}$ are applied
only in the space variable  $x$.)

\medskip

This differs from the corresponding definition in \cite{BLL} in omitting
the $\lambda e_1$ term from the formula for $u$; this change represents viewing
the flow in a frame stationary relative to the fluid at infinity instead of one
moving with velocity $-\lambda e_1$.

For a discussion of the global existence of $L^p$-regular solutions of the
vorticity equation with given initial vorticity in $L^p(\Pi) \cap L^1(\Pi)$
see \cite[Sect. 2]{BLL}.
Note that solutions are not known to be unique, except in the case of compactly
supported initial vorticity in $L^\infty$ studied by Yudovich \cite{VIY:NSF}.

In stating the following theorem, we regard vorticity $\omega$ a function
of time $t$ with values in $L^1(\Pi) \cap L^p(\Pi)$ and suppress the space
variable $x$.

\begin{theorem}[{\bf Stability}]
Let $2<p<\infty$, let $\zeta_0 \in L^1(\Pi) \cap L^p(\Pi)$ be non-negative
and have compact support and let $i_0>0$. 
Suppose $\emptyset \neq \Sigma_0 \subset \sR(\zeta_0)$.
Then for every $\varepsilon>0$ there exists $\delta>0$ such that, if
$\omega(0) \geq 0$ with compact support satisfies
$\mathrm{dist}_\Xp (\omega(0),\Sigma_0)<\delta$,
then
$\mathrm{dist}_\Xp (\omega(t),\Sigma_0)<\varepsilon$
for all $t>0$, whenever $\omega$ is an $L^p$-regular solution of
the vorticity equation with initial vorticity $\omega(0)$.
\label{thm2}
\end{theorem}

The proofs will be given in Section \ref{S4}.

\subsection{Estimates for the stream function and properties of the energy.}

In Lemmas \ref{lm1}, \ref{lm3} and \ref{lmx1} following we refine the
calculations of \cite[Lemma 1,2,10]{GRB:VP}
to derive estimates for $\sG(\zeta)$ that apply uniformly over
$\zeta \in \sW(\zeta_0, \leq i_0)$ and in Lemmas \ref{lmx5}, \ref{lmx2}
and \ref{lmx4} we establish weak continuity properties of $E$.
Lemmas \ref{lmx8} and \ref{lm2} then show we can work in a strip in $\Pi$.

\noindent
\begin{lemma}
\label{lm1}
Suppose that $1<p<\infty$ and $1<\alpha<\infty$.
Then there are positive constants $c_1,c_2,c_3$, depending only on $p$ and
$\alpha$, such that, if
$0 \leq \zeta \in L^p(\Pi)$ and $I(\zeta)<\infty$, then
\[
\sG\zeta(x) \leq x_2^{-1}(c_1 \log x_2 + c_2) I(\zeta)
+ c_3 x_2^{-1/\alpha} \| \zeta \|_p^{1-1/\alpha} I(\zeta)^{1/\alpha}
\quad \forall x_2 \geq 1.
\]
\end{lemma}

\noindent
\begin{proof}
We have
\[
\sG\zeta(x) = \left(\int_{\rho<x_2/2} + \int_{\rho>x_2/2}\right)
G(x,y)\zeta(y) \rd y 
\]
where $\rho=|x-y|$.
Now
\begin{eqnarray*}
\int_{\rho>x_2/2} G(x,y)\zeta(y) \rd y &=&
(4\pi)^{-1} \int_{\rho>x_2/2}\log(1+ 4x_2 y_2 \rho^{-2}) \zeta(y) \rd y\\
& \leq & \pi ^{-1} \int_{\rho>x_2/2} x_2 y_2\rho^{-2} \zeta(y) \rd y
\leq  \pi ^{-1} \int_{\rho>x_2/2} 4 x_2^{-1} y_2 \zeta(y) \rd y \\
& \leq &
4\pi^{-1} I(\zeta) x_2^{-1}.
\end{eqnarray*}
For $0<\rho<x_2/2$ we have $x_2/2 < y_2 < 3x_2/2$ so
\begin{eqnarray*}
4\pi G  (x,y) &=& \log(1+4x_2 y_2/\rho^2) < \log (9x_2 y_2/2\rho^2)\\
&=&
\log(9y_2/2) + \log x_2 + 2 \log \rho^{-1}\\
& \leq &
2x_2^{-1}y_2\log(27x_2/4) + 2 x_2^{-1}y_2\log x_2 + 2 \log\rho^{-1}
\end{eqnarray*}
when $x_2 \geq 1$, hence
\[
\int_{\rho<x_2/2} G(x,y)\zeta(y) \rd y  \leq 
(4\pi x_2)^{-1}\left(4\log x_2+ 2\log{\textstyle \frac{27}{4}}\right)I(\zeta)
+ 2\int_{\rho<x_2/2} (\log\rho^{-1})\zeta(y) \rd y .
\]
To treat the last integral in the above inequality, choose
$\beta = p/(1-1/\alpha)$, which ensures that $\beta>p>1$ and that
\[
\frac{1}{\alpha} + \frac{1}{\beta}
= \frac{1}{\alpha} + \frac{1}{p}\left( 1 -\frac{1}{\alpha} \right)
< \frac{1}{\alpha} + \left( 1 -\frac{1}{\alpha} \right) = 1 ,
\]
hence we can choose $\gamma>1$ such that
$\alpha^{-1} + \beta^{-1} + \gamma^{-1} =1$.
Then H\"older's inequality yields
\begin{eqnarray*}
\int_{\rho<x_2/2} (\log\rho^{-1}) \zeta(y)\rd y
&\leq&
\int_{\rho<x_2/2} (\log\rho^{-1})_+ \, y_2^{-1/\alpha}y_2^{1/\alpha}\zeta(y)
\rd y\\
&\leq&
2^{1/\alpha} x_2^{-1/\alpha} \int_{\rho<x_2/2} (\log\rho^{-1})_+
\zeta^{1-1/\alpha}y_2^{1/\alpha} \zeta^{1/\alpha} \rd y\\
&\leq&
2^{1/\alpha} x_2^{-1/\alpha} \| (\log \rho^{-1})_+ \|_\gamma
\| \zeta^{1-1/\alpha} \|_\beta \| y_2^{1/\alpha} \zeta^{1/\alpha} \|_\alpha\\
&\leq&
2^{1/\alpha} x_2^{-1/\alpha} \| (\log \rho^{-1})_+ \|_\gamma
\| \zeta\|_p^{p/\beta} I( \zeta)^{1/\alpha}\\
&\leq&
2^{1/\alpha} x_2^{-1/\alpha} \| (\log \rho^{-1})_+ \|_\gamma
\| \zeta\|_p^{1-1/\alpha} I( \zeta)^{1/\alpha}
\end{eqnarray*}
and we note that 
\[
\| (\log \rho^{-1})_+ \|_\gamma =
\left( \int_0^1 2\pi (\log \rho^{-1})^\gamma \rho \rd  \rho \right)^{1/\gamma}
\]
which is a positive real number depending only on $\alpha$ and $p$ since
$\gamma=\alpha p / ((\alpha -1)(p-1))$.
\end{proof}

\medskip

\begin{lemma}
Suppose that $1<p<\infty$ and $1/p+1/q=1$.
Then there is a positive constant $c_4$, depending only on $p$, such that, if
$0 \leq \zeta \in L^p(\Pi)$, then
\[
\sG\zeta(x) \leq c_4 (\|\zeta\|_1+\|\zeta\|_p) (x_2^{1/q} + x_2^2)
\quad \forall x \in \Pi.
\]
\label{lm3}
\end{lemma}

\begin{proof}
Write $\rho=|x-y|$.
Then $y_2 \leq x_2+\rho$ so by \eqref{eqx6} and concavity of $\log$
we have firstly
\[
\int_{\rho > 1} G(x,y)\zeta(y)\rd y
\leq \int_{\rho > 1} \frac{x_2 y_2}{\pi \rho^2} \zeta(y)\rd y
\leq \int_{\rho > 1} \frac{x_2^2 +x_2\rho}{\pi \rho^2} \zeta(y)\rd y
\leq \frac{1}{\pi} (x_2^2+x_2) \| \zeta \|_1 .
\]
Secondly, if $\rho<x_2$ then $|x-\overline{y}| \leq 2x_2 + \rho \leq 3x_2$, so
we have
\[
\int_{\rho<x_2} G(x,y)\zeta(y) \rd y
\leq \frac{1}{2\pi}
\int_{\rho<x_2} \log\left(\frac{3x_2}{\rho}\right)\zeta(y)\rd y
\leq \frac{1}{2\pi}
\left(2\pi x_2^2 \int_0^1 (\log(3s^{-1}))^q \rd s \right)^{1/q} \|\zeta\|_p
\]
by H\"older's inequality.
Thirdly, when $x_2<1$ the region $x_2<\rho<1$ is nonempty
and \eqref{eqx6} yields
\begin{eqnarray*}
\int_{x_2 < \rho <1} G(x,y)\zeta(y)\rd y
&\leq& \int_{x_2 < \rho <1} \frac{x_2(x_2+\rho)}{\pi\rho^2} \zeta(y)\rd y
\leq \int_{x_2 < \rho <1} \frac{2x_2}{\pi\rho} \zeta(y) \rd y \\
&\leq& \frac{2}{\pi} \int_{x_2<\rho<1}
\left(\frac{x_2}{\rho}\right)^{1/q} \zeta(y) \rd y 
\leq \frac{4\|\zeta\|_p}{(2\pi)^{1/p}} x_2^{1/q}
\end{eqnarray*}
by H\"older's inequality.
These three bounds yield the result since
$x_2 + x_2^{2/q} \leq 2(x_2^{1/q}+x_2^2)$.
\end{proof}

\medskip

\noindent
{\bf Remark.} When $2<p<\infty$ we can strengthen Lemma \ref{lm3} as follows:

\medskip

\begin{lemma}
Let $2<p<\infty$ and $1/q+1/p=1$.
Then there is a constant $c_5>0$ depending only on $p$, such that if 
$0 \leq \zeta \in L^p(\Pi) \cap L^1(\Pi)$ then, for $x \in \Pi$
\begin{eqnarray*}
|\nabla_x \sG \zeta(x)| &\leq& c_5 (\|\zeta\|_p+\|\zeta\|_1), \\
\sG \zeta(x) &\leq& c_5 x_2 (\|\zeta\|_p + \|\zeta\|_1) .
\end{eqnarray*}
\label{lmx1}
\end{lemma}

\begin{proof}
Let $0 \leq \zeta \in L^p(\Pi) \cap L^1(\Pi)$.
Then
\begin{eqnarray*}
\int_\Pi |\nabla_x G(x,y) \zeta(y)| \rd y
&=& \frac{1}{2\pi}\int_\Pi \left| \frac{x-y}{|x-y|^2}
- \frac{x-\overline{y}}{|x-\overline{y}|^2} \right| \zeta(y) \rd y \\
&\leq& \frac{1}{2\pi}\int_\Pi \left( \frac{1}{|x-y|}
+ \frac{1}{|x-\overline{y}|} \right) \zeta(y) \rd y \\
&\leq& \frac{1}{\pi}\int_\Pi \frac{1}{|x-y|} \zeta(y) \rd y \\
&=& \frac{1}{\pi}\left( \int_{D_\Pi(x,1)} + \int_{\Pi \setminus D(x,1)} \right)
\frac{1}{|x-y|} \zeta(y) \rd y \\
&\leq& \frac{1}{\pi}
\left( \int_0^1 \frac{2\pi\rho \rd\rho}{\rho^q} \right)^{1/q}
\| \zeta \|_p + \frac{1}{\pi} \| \zeta \|_1 ,
\end{eqnarray*}
hence the first inequality.
The second inequality now follows since $\sG\zeta(x) \to 0$
as $x_2 \to 0$ uniformly over $x_1$ by Lemma \ref{lm3}.
\end{proof}

\begin{lemma}
Let $1<p<\infty$, let $L>0$ and let 
\[
\sL = \{ \xi \in L^1(\Pi) \cap L^p(\Pi) \mid \xi \geq 0, \; I(\xi) \leq L, \;
\| \xi \|_1 \leq L, \; \|\xi\|_p \leq L \}.
\]
Then $E$ is Lipschitz continuous with respect to
$\| \, \|_1$ relative to $\sL$.
\label{lmx5}
\end{lemma}

\begin{proof}
From Lemmas \ref{lm1} and \ref{lm3} we can choose a constant $K$ such that
$\| \sG\xi \|_\mathrm{sup} \leq K$ for all $\xi \in \sL$.
Consider $\xi,\eta \in \sL$.
We can assume $E(\xi) \geq E(\eta)$ and so
\[
0 \leq E(\xi)-E(\eta)
= \frac{1}{2} \int_\Pi (\xi-\eta)\sG(\xi+\eta) 
\leq \frac{1}{2}\|\xi-\eta\|_1(\|\sG\xi\|_\mathrm{sup}+\|\sG\eta\|_\mathrm{sup})
\leq K\|\xi-\eta\|_1  .
\]
\end{proof}

\begin{lemma}
Let $2<p<\infty$. Then there is a constant $c_6>0$ such that
\[
\sG\zeta(x) \leq
c_6(I(\zeta)+\|\zeta\|_1+\|\zeta\|_p) x_2 \min\{1,| x_1 |^{-1/(2p)}\},
\quad x \in \Pi
\]
for all non-negative $\zeta \in L^p(\Pi) \cap L^1(\Pi)$ that are
Steiner-symmetric in the $x_1$-axis with $I(\zeta)<\infty$.
\label{lmx2}
\end{lemma}

\begin{proof}
Firstly note that if $\xi \in L^1(0,\infty)$ is decreasing and $0<b<s$ then
the means of $\xi$ on both of the intervals $[s-b,s]$ and $[s,s+b]$ are
no greater than the mean of $\xi$ on $[0,s]$ and therefore
\[
\frac{1}{2b} \int_{s-b}^{s+b} \xi \leq \frac{1}{s} \int_0^s \xi .
\]

Now consider Steiner-symmetric $\zeta\in L^1(\Pi)$ and $x \in \Pi$.
Let $0<b \leq |x|_1$ and apply the above inequality in the $y_1$ integration 
to obtain
\[
\frac{1}{2b} \int_{|y_1-x_1|<b} \zeta(y)\rd y
\leq \frac{1}{2|x_1|} \int_{|y_1|<|x_1|} \zeta(y)\rd y
\]
and consequently
\begin{equation}
\int_{|y_1-x_1|<b} \zeta(y)\rd y \leq \frac{b}{|x_1|} \int_\Pi \zeta(y)\rd y .
\label{eqx7}
\end{equation}

Suppose that additionally $I(\zeta)<\infty$, fix $x \in \Pi$ with
$x_1 \neq 0$, let
\[
\zeta_1(y) =
\begin{cases}
\zeta(y) & \mbox{ if } |y_1-x_1|  <   |x_1|^{1/2} \\
0        & \mbox{ if } |y_1-x_1| \geq |x_1|^{1/2}
\end{cases}
\]
and let $\zeta_2 = \zeta - \zeta_1$.
We write $\rho=|x-y|$ when $x,y \in \Pi$ and obtain
\begin{eqnarray*}
\sG\zeta_2(x)
&\leq& \frac{1}{4\pi} \int_{\rho>|x_1|^{1/2}}
\log\left(1+\frac{4x_2 y_2}{\rho^2}\right)\zeta_2(y)\rd y
\leq \frac{1}{4\pi} \int_{\rho>|x_1|^{1/2}}
\frac{4x_2 y_2}{\rho^2}\zeta_2(y)\rd y \\
&\leq& \frac{x_2}{\pi|x_1|} \int_\Pi \zeta(y) y_2 \rd y
= \frac{x_2}{\pi|x_1|} I(\zeta) .
\end{eqnarray*}
When $|x_1|>1$ we have from \eqref{eqx7}
\begin{eqnarray*}
\int_\Pi \zeta_1^1 &\leq& \frac{|x_1|^{1/2}}{|x_1|} \int_\Pi \zeta
= |x_1|^{-1/2} \int_\Pi \zeta,  \\
\int_\Pi \zeta_1^p &\leq& \frac{|x_1|^{1/2}}{|x_1|} \int_\Pi \zeta^p
= |x_1|^{-1/2} \int_\Pi \zeta^p
\end{eqnarray*}
so by Lemma \ref{lmx1} we have
\[
\sG\zeta_1(x) \leq c_5 x_2 (\|\zeta_1\|_1 + \|\zeta_1\|_p)
\leq c_5 x_2 (\|\zeta\|_1 |x_1|^{-1/2} + \|\zeta\|_p |x_1|^{-1/(2p)}).
\]
The result follows from the above inequalities for $\sG\zeta_1(x)$ and
$\sG\zeta_2(x)$ when $|x_1|>1$ and from Lemma \ref{lmx1} when $|x_1| \leq 1$.
\end{proof}

\begin{lemma}
Let $2<p<\infty$ and let $(\zeta_n)$ be a sequence of Steiner symmetric
non-negative functions on $\Pi$, suppose $\| \zeta_n \|_1$ and $I(\zeta_n)$
are bounded and suppose $\zeta_n \to \zeta$ weakly in $L^p(\Pi)$.
Then $E(\zeta_n) \to E(\zeta)$ as $n \to \infty$.
\label{lmx4}
\end{lemma}

\begin{proof}
Fix $T>0$, to be chosen later, and consider Steiner symmetric
$\xi \in L^1(\Pi) \cap L^p(\Pi)$ with $I(\xi)<\infty$.
Define
\begin{eqnarray*}
\Pi_0 &=& \{ x \in \Pi \mid x_2>T \}, \\
\Pi_1 &=& \{ x \in \Pi \mid x_2<T, \; |x_1|>T \}, \\
\Pi_2 &=& \{ x \in \Pi \mid x_2<T, \; |x_1|<T \} 
\end{eqnarray*}
and let $\xi_k = \xi 1_{\Pi_k}$ for $k=1,2,3$.
Then we have
\begin{equation}
E(\xi)
= \frac{1}{2} \int_\Pi \xi_2 \sG \xi_2 + \int_\Pi \xi_2 \sG(\xi_0 + \xi_1)
+ \frac{1}{2} \int_\Pi (\xi_0+\xi_1) \sG(\xi_0 + \xi_1) .
\label{eqx3}
\end{equation}
There is a positive constant $C$, depending only on $p$, such that
the inequalities
\begin{eqnarray}
\sG\xi(x) &\leq& C(I(\xi) + \|\xi\|_1 + \|\xi\|_p) \label{eq3} \\
\sG\xi(x) &\leq& 
C(I(\xi)+\|\xi\|_1+\|\xi\|_p) x_2\min\{1,|x_1|^{-1/(2p)}\} \label{eq4} 
\end{eqnarray}
hold for all $x \in \Pi$, from Lemmas \ref{lm1} and \ref{lm3}
in the case of \eqref{eq3} and from Lemma \ref{lmx2} in the case of
\eqref{eq4}, the Steiner symmetry being employed only for \eqref{eq4}.
From \eqref{eq3} and \eqref{eq4} we obtain
\begin{eqnarray}
\int_\Pi \xi_2 \sG(\xi_0 + \xi_1)
&+&
\frac{1}{2} \int_\Pi (\xi_0+\xi_1) \sG(\xi_0 + \xi_1)
\leq \int_\Pi \xi_0\sG\xi + \xi_1\sG\xi \notag \\
&\leq&
\|\sG\xi\|_\mathrm{sup} \int_{x_2>T} \xi + \int_{|x_1|>T} \xi\sG\xi \notag \\ 
&\leq&
C(I(\xi)+\|\xi\|_1+\|\xi\|_p)
\left( \int_{x_2>T} T^{-1}x_2\xi(x) \rd x
+ \int_{|x_1|>T} x_2|x_1|^{-1/(2p)}\xi(x)\rd x \right) \notag\\
&\leq&
C(I(\xi)+\|\xi\|_1+\|\xi\|_p)I(\xi)(T^{-1}+T^{-1/(2p)}) .  \label{eqx4}
\end{eqnarray}

Now consider the sequence $(\zeta_n)$, which must be bounded in $L^p(\Pi)$.
Let $\varepsilon>0$, write $Q = (-T,T) \times (0,T)$ and use \eqref{eqx3}
and \eqref{eqx4} to choose $T>0$ such that
\[
0 \leq E(\zeta_n) - E(1_Q \zeta_n) < \varepsilon
\]
for all $n$ and
\[
0 \leq E(\zeta) - E(1_Q \zeta) < \varepsilon.
\]
In view of the compactness of $\sG$ as an operator from $L^p(Q)$ to
$L^q(Q)$ we have
\[
E(1_Q \zeta_n) \to E(1_Q \zeta)
\]
as $n \to \infty$. It follows that
\[
E(\zeta_n) \to E(\zeta).
\]
\end{proof}

\begin{lemma}
Let $1<p<\infty$, let $i_0 >0$ and let $\zeta_0$ be non-negative and have
compact support.
Let $\zeta \in \sW(\zeta_0,\leq i_0)$.
Then $\zeta$ is the weak limit in $L^p$ of a sequence $(\zeta_n)$ in
$\sR^+(\zeta_0)$ having $I(\zeta_n) \leq i_0$ for each $n$.
\label{lmx8}
\end{lemma}

\begin{proof}
Firstly consider the case when $I(\zeta)<i_0$.
Since $\sR(\zeta_0)$ is weakly dense in $\sW(\zeta_0)$ by the results of
Douglas \cite{RJD:UD} we may choose a sequence $(\xi_k)$ in $\sR(\zeta_0)$
converging weakly to $\zeta$ in $L^p(\Pi)$.

Given $g \in L^q(\Pi)$, where $1/p+1/q=1$, we have
\[
\left| \int_{D_\Pi(0,n)} \xi_k g - \int_\Pi \zeta g \right| 
\leq
\left| \int_\Pi \xi_k g - \int_\Pi \zeta g \right| 
+ \| \zeta_0 \|_p \| g \|_{L^q(\Pi \setminus D(0,n))}
\]
for all $k,n \in \NN$.
Hence if $(k(n))$ is any strictly increasing sequence in $\NN$ then
$\xi_{k(n)} 1_{D(0,n)} \to \zeta$ weakly in $L^p$ as $n \to \infty$.

Further, for each fixed $n \in \NN$ we have
$I(\xi_k 1_{D(0,n)}) \to I(\zeta 1_{D(0,n)})$ as $k \to \infty$, so by
a diagonal sequence argument we can chose a strictly increasing sequence
$(k(n))$ of positive integers such that
$I(\xi_{k(n)} 1_{D(0,n)}) - I(\zeta 1_{D(0,n)}) \to 0$ as $n \to \infty$.
Since $I(\zeta 1_{D(0,n)}) \to I(\zeta)$ by the monotone convergence theorem
we now have $I(\xi_{k(n)} 1_{D(0,n)}) \to I(\zeta)$ as $n \to \infty$.

Therefore $\xi_{k(n)} 1_{D(0,n)} \in \sR^+(\zeta_0)$ satisfies
$I(\xi_{k(n)} 1_{D(0,n)}) < i_0$ for all sufficiently large $n$ and
$\xi_{k(n)} 1_{D(0,n)} \to \zeta$ weakly in $L^p$.
This completes the proof in the case $I(\zeta)<i_0$.

Secondly, suppose $I(\zeta)=i_0>0$.
Then, by truncation, $\zeta$ can we written as the strong
limit of elements $\xi_n =\zeta 1_{\{x \mid x_2<r_n\}}$ of $\sW(\zeta_0)$
with $I(\xi_n)<i_0$, where $(r_n)$ is an increasing sequence of positive
numbers,  so each $\xi_n$ is in the weak closure
of $\sR^+(\zeta_0) \cap I^{-1}([0,i_0])$ by the above
argument, hence $\zeta$ is in the weak closure of
$\sR^+(\zeta_0) \cap I^{-1}([0,i_0])$.
\end{proof}

\begin{lemma}
Let $i_0>0$ and $0 \leq \zeta_0 \in L^p(\Pi)$ for some $1<p<\infty$,
compactly supported.
Then there exists $Z>0$ such that, if $(\zeta_n)$ is any maximising
sequence for $E$ relative to $\sW(\zeta_0,i_0)$ comprising elements of
$\sR^+(\zeta_0)$ then ${\displaystyle \left(\zeta_n 1_{\RR \times (0,Z)}\right)}$
is also a maximising sequence and 
${\displaystyle \left\|\zeta_n 1_{\RR\times(Z,\infty)}\right\|_1 \to 0}$.
\label{lm2}
\end{lemma}

\begin{proof}
Let $M_0$ be the supremum of $E$ relative to $\sW(\zeta_0,\leq i_0)$,
so $M_0<\infty$ by Lemma \ref{lmx5}.
Consider $\zeta \in \sW(\zeta_0,\leq i_0) \cap \sR^+(\zeta_0)$ such that
$E(\zeta) \geq M_0/2$; such a $\zeta$ must exist if any such sequences
$(\zeta_n)$ exist.
Write $S=\{x \mid \zeta(x)>0\}$ and $S_0 = \{x \mid \zeta_0(x)>0\}$.

Then $\| \zeta \|_1 \leq \| \zeta_0 \|_1$ since $\zeta \in \sR^+(\zeta_0)$,
hence $\| \sG \zeta \|_{\sup} \geq M_0/\| \zeta_0 \|_1$
by H\"{o}lder's inequality.
The estimate of Lemma \ref{lm1} shows there exists $Z_0>0$ such that
every function $\xi \in \sW(\zeta_0,\leq i_0)$ obeys
$\sG\xi(x)< M_0/(2\|\zeta_0\|_1)$ if $x_2>Z_0$.
Choose $x^* \in \Pi$ at which $\sG \zeta$ achieves its supremum;
thus $x^*_2 \leq Z_0$.

Consider $y \in \Pi$ satisfying $y_2>2Z_0$.
Then we have
$|x^*-y| \geq y_2-x^*_2 > y_2/2$ so
\[
G(x^*,y) =
\frac{1}{4\pi} \log \left( 1+ \frac{4 x^*_2 y_2}{|x^*-y|^2} \right)
< \frac{1}{4\pi} \log \left( 1+ \frac{16 Z_0}{y_2} \right) .
\]
We choose $Z_1>2Z_0$ such that $Z_1^2 \geq |S_0|$ and
\[
\frac{1}{4\pi} \log \left( 1+ \frac{16 Z_0}{Z_1} \right)
\leq \frac{M_0}{2\| \zeta_0\|_1^2} .
\]
This choice of $Z_1$ ensures that if $y_2>Z_1$ then we have
\[
G(x^*,y) < \frac{1}{4\pi} \log \left( 1+ \frac{16 Z_0}{Z_1} \right)
\leq \frac{M_0}{2\| \zeta_0\|_1^2}
\]
whereas if $y_2 \leq Z_1$ and $|y_1-x^*_1|>Z_1$ then we have
\[
G(x^*,y) < \frac{1}{4\pi} \log \left( 1+ \frac{4 Z_0}{Z_1} \right)
\leq \frac{M_0}{2\| \zeta_0\|_1^2} .
\]

Write $Q=(x^*_1-Z_1,x^*_1+Z_1) \times (0,Z_1)$; thus
$|Q| = 2 Z_1^2 \geq 2|S_0|$.
Then, by the above bounds,
\[
\int_{\Pi \setminus Q} G(x^*,y) \zeta(y) \rd y \leq \frac{M_0}{2\|\zeta_0\|_1}
\]
and therefore, since the choice of $x^*$ ensures
$\sG(x^*) \geq M_0/\|\zeta_0\|_1$, we have
\[
\int_Q  G(x^*,y) \zeta(y) \rd y \geq \frac{M_0}{2\|\zeta_0\|_1}.
\]
It follows by H\"older's inequality that
\[
\int_Q G(x^*,y) \zeta(y) \rd y
\leq \|G(x^*,\cdot)\|_{L^q(Q)} \|\zeta\|_{L^p(Q)}
\leq g \|\zeta\|_{L^p(Q)}
\]
where $1/q+1/p=1$ and we may find a suitable value for the constant $g$
from the rearrangement inequality for the integral of a product of two
functions, as follows.
We have
\[
G(x^*,y) \leq \frac{1}{4\pi} \log \left( 1+\frac{8Z_1^2}{|x^*-y|^2} \right),
\quad y \in Q,
\]
so, denoting by $D$ the disc with centre $x^*$ and radius $Z_1$,
denoting by $Q^*$ the disc with centre $x^*$ and area $|Q|$ and noting that
$|D| \geq |Q^*|$, we have
\begin{eqnarray*}
\|G(x^*,\cdot)\|_{L^q(Q)}^q = \int 1_Q(y) G(x^*,y)^q \rd y 
& \leq & \int 1_{Q^*}(y) 
\left(\frac{1}{4\pi}\log\left(1+\frac{8Z_1^2}{|x^*-y|^2}\right)\right)^q \rd y
\\
& \leq & \int_D
\left(\frac{1}{4\pi}\log\left(1+\frac{8Z_1^2}{|x^*-y|^2}\right)\right)^q \rd y
\\
& \leq & \int_0^{Z_1} 
\left(\frac{1}{2\pi}\log\left(\frac{3Z_1}{\rho}\right)\right)^q 2\pi\rho \rd\rho
=: g^q .
\end{eqnarray*}
Hence
\[
\| \zeta \|_{L^p(Q)} \geq  \frac{M_0}{2g\|\zeta_0\|_1} =:m .
\]
Choose $\theta>0$ such that
\[
\int_0^\theta (\zeta_0^\Delta)^p <  \frac{m^p}{2} .
\]
Then we have
\[
|Q \cap S | \geq 2\theta .
\]
Let $Q_0=(x^*_1 -Z_1,x^*_1 +Z_1)\times(0,\eta)$ where
$\eta = \theta/(2Z_1)$. 
Then $|Q_0|=\theta$ and therefore
\[
\int_{Q \setminus Q_0} \zeta
\geq \inf_{U \subset S, |U|=\theta} \int_U \zeta
\geq \inf_{U \subset S_0, |U| = \theta} \int_U \zeta_0
\geq \int_0^\theta \zeta_0^\nabla =: \nu .
\]
If $x,y \in Q \setminus Q_0$ then
\[
G(x,y) \geq \frac{1}{4\pi}\log\left(1+\frac{4\eta^2}{5Z_1^2}\right) =:\mu .
\]
Hence for all $x \in Q \setminus Q_0$ we have
$\sG\zeta(x) \geq \mu \nu$.
Moreover
\[
|Q \setminus Q_0| = |Q|-\theta \geq \frac{|Q|}{2} \geq |S_0| .
\]

Choose $Z > Z_1$ such that $x_2>Z$ implies
$\sG\xi(x)<\mu\nu/2$ for all $\xi \in \sW(\zeta_0,\leq i_0)$,
by Lemma \ref{lm1}.
Let $h = \zeta 1_{\RR \times (Z,\infty)}$ and consider the possibility that
$h$ is non-trivial; then $h^{-1}(0,\infty)$ is a set of finite positive planar
Lebesgue measure and is therefore measure-theoretically isomorphic to a bounded
interval in $\RR$ with linear Lebesgue measure.
Similarly $Q \setminus (Q_0 \cup S)$, which has planar Lebesgue
measure greater than that of $S \setminus Q$, is therefore
isomorphic to another interval, of greater length.
It follows that we can choose a rearrangement $h^\prime$ of $h$ supported in
$Q \setminus (Q_0 \cup S)$.
Then $\zeta +h^\prime -h$ lies in $\sW(\zeta_0,\leq i_0) \cap \sR^+(\zeta_0)$
and 
\[
E(\zeta +h^\prime -h)
= E(\zeta) + \int_\Pi (\sG \zeta) (h^\prime - h) + E(h^\prime - h)
\geq E(\zeta) + \mu\nu\|h^\prime\|_1 - \frac{\mu\nu}{2}\|h\|_1
= E(\zeta) + \frac{\mu\nu}{2} \| h \|_1,
\]
so
\[
\| h \|_1 \leq \frac{2}{\mu\nu}(M_0 - E(\zeta)).
\]
Hence if $(\zeta_n)$ is a maximising sequence belonging to
$\sW(\zeta_0,\leq i_0) \cap \sR^+(\zeta_0)$ and
$h_n =  \zeta_n 1_{\{x \in \Pi \mid x_2>Z \}}$ then $\| h_n \|_1 \to 0$ so
\[
E(\zeta_n - h_n) = E(\zeta_n) - \int_\Pi (\sG\zeta_n)h_n +E(h_n)
\geq E(\zeta_n) - \mu\nu \| h_n \|_1 \to M_0.
\]
That is, $( \zeta_n 1_{\RR \times (0,Z)})$ is also a maximising sequence of
functions in $\sW(\zeta_0,\leq i_0) \cap \sR^+(\zeta_0)$.
\end{proof}

\section{Existence of relaxed maximisers and first variation condition.}
\label{exist}

In this section we use the estimates of Section \ref{prelim} to extend the
existence theory of \cite[Theorem 16(ii)]{GRB:VP} to the relaxed problem,
since in the proofs of the main results we will have to consider subsidiary
variational problems where only relaxed solutions might exist under the
hypotheses that apply
(situations where solutions to an unrelaxed problem of this type fail
to exist can be found
in the study of Lamb's vortex \cite[Corollary 1]{GRB:LAMB}).
The first variation condition at a maximum gives rise to a functional
relationship between the vorticity and the stream function that shows the
maximisers represent steady flows.
Finally we show that for large values of the impulse, the relaxed solutions
are indeed solutions of the unrelaxed problem.
This will be needed to show that the Stability Theorem applies in a wide
range of cases.

\begin{lemma}
Let $2<p<\infty$, let $i_0>0$ and let
$0 \leq \zeta_0 \in L^1(\Pi) \cap L^p(\Pi)$ have compact support.
Then\\
{\rm (i)} there exist maximisers for $E$ relative to $\sW(\zeta_0,\leq i_0)$
and all maximisers belong to $\sW(\zeta_0,i_0)$,\\
{\rm (ii)} all maximisers are Steiner-symmetric elements of $\sR\sC(\zeta_0)$,\\
{\rm (iii)} for every maximiser $\zeta$ there exists an increasing function
$\varphi$ and a number $\lambda > 0$ such that
$\zeta = \varphi \circ (\sG\zeta - \lambda x_2)$ almost everywhere in $\Pi$
and $\zeta$ vanishes almost everywhere in the set
$\{ x \in \Pi \mid \psi(x)-\lambda x_2 \leq  0 \}$,\\
{\rm (iv)} every maximiser vanishes outside a bounded subset of
$\RR \times (0,Z)$, where $Z$ is the number, depending on $i_0$, $\zeta_0$
and $p$ only, provided by Lemma \ref{lm2}.
\label{lmx6}
\end{lemma}

\begin{proof}
To prove (i) we choose a maximising sequence $(\zeta_n)$ for $E$ relative
to $\sW(\zeta_0, \leq i_0)$.
It follows from the 1-dimensional case of the Riesz rearrangement inequality
that Steiner symmetrisation about the $x_2$-axis does not decrease $E$.
Therefore let us assume $(\zeta_n)$ to comprise Steiner-symmetric functions.
The sequence $(\zeta_n)$ is bounded in both $\| \, \|_1$ and $\| \, \|_p$.
We may therefore pass to a subsequence and assume $(\zeta_n)$ converges weakly
in $L^p(\Pi)$ to a limit $\widehat{\zeta}$.
Then $\widehat{\zeta} \in \sW(\zeta_0,\leq i_0)$ and $E(\widehat{\zeta}) = M_0$
by Lemma \ref{lmx4}.
Thus $\widehat{\zeta}$ is a maximiser.

If $\zeta \in \sW(\zeta_0, \leq i_0)$ is any element then translation of
$\zeta$ in the positive $x_2$ direction strictly increases $E(\zeta)$.
Therefore every maximiser $\zeta$ of $E$ relative to $\sW(\zeta_0,\leq i_0)$
must satisfy $I(\zeta)=i_0$.

To prove (ii) consider a maximiser $\zeta$.
We can write
\[
E(\zeta) = \int_0^\infty \int_0^\infty J(\zeta,x_2,y_2) \rd x_2 \rd y_2
\]
where
\[
J(\zeta,x_2,y_2) = \int_{-\infty}^\infty \int_{-\infty}^\infty
\zeta(x_1,x_2)G(x_1,x_2,y_1,y_2)\zeta(y_1,y_2)\rd x_1 \rd y_1 
\]
with similar expressions for $E(\zeta^s)$ and $J(\zeta^s,x_2,y_2)$.
Since $G(x_1,x_2,y_1,y_2)$ is a decreasing function of $|x_1-y_1|$ for
fixed $x_2$ and $y_2$, the Riesz rearrangement inequality shows that
\[
J(\zeta,x_2,y_2) \leq J(\zeta^s,x_2,y_2) \quad \forall x_2>0,y_2>0
\]
and hence $E(\zeta) \leq E(\zeta^s)$, so $E(\zeta)=E(\zeta^s)$ by maximality.
Thus
\[
\int_0^\infty \int_0^\infty (J(\zeta^s,x_2,y_2)-J(\zeta,x_2,y_2))
\rd x_2 \rd y_2 =0
\]
and since the integrand is non-negative we must have
\[
J(\zeta,x_2,y_2)=J(\zeta^s,x_2,y_2) < \infty
\]
for almost all pairs $(x_2,y_2)$ of positive numbers. 
We can now apply the one-dimensional case of Lieb's analysis\footnote
{Lieb's result applies to the case when one of the functions is strictly
symmetric decreasing, which is the situation here, and has an unstated but
necessary assumption that the integrals in question are finite.
In the proof, $m$ should be defined by $m:=f*\overline{h}$, where
$\overline{h}(y)=h(-y)$, so that $m$ is invariant under equal translations
of $f$ and $h$.}
\cite[Lemma 3]{LIEB}
of equality in the Riesz rearrangement inequality to conclude that,
for almost every pair $(x_2,y_2)$ of positive real numbers,
the functions $\zeta(\cdot,x_2)$ and $\zeta(\cdot,y_2)$ are symmetric
decreasing about the same real number $\beta$ say, and then that $\beta$ is
independent of $x_2$ and $y_2$.
Thus $\zeta$ is Steiner symmetric about the line $x_2 = \beta$.

Lemmas \ref{lmx8} and \ref{lm2} show that $\zeta$ is a weak limit in $L^p$ of
functions $\zeta_n \in \sR^+(\zeta_0)$ satisfying $I(\zeta_n) \leq i_0$ and that
$\zeta_n 1_{\RR \times (Z,\infty)} \to 0$ in $L^1$.
It follows that $\zeta$ vanishes outside $\RR \times (0,Z)$.
Any maximiser $\zeta$ satisfies $I(\zeta)=i_0$ and
the strict convexity of $E$ shows that $\zeta$ must be an
extreme point of $\sW_{\RR \times (0,Z)} (\zeta_0,i_0)$.
Since $I$ is a bounded linear functional relative to $L^1(\RR \times (0,Z))$,
we can apply Douglas's characterisation \cite[Theorem 2.1(ii)]{RJD:UD} of
the extreme points of the intersection of $\sW(\zeta_0)$
with a closed hyperplane to deduce that $\zeta \in \sR\sC(\zeta_0)$
(the statement of Douglas's result excludes $L^1$ but the proof of part (ii)
is also valid for $L^1$).

To prove  (iii) and (iv) let $\zeta$ be a maximiser and $\psi = \sG \zeta$.
From convexity of $E$ it follows that $\zeta$ maximises $\int_\Pi \psi\xi$
subject to $\xi \in \sW(\zeta_0,\leq i_0)$.
Define the ``value function'' $f$ by
\[
f(i) = \sup_{\xi \in \sW(\zeta_0,\leq i)} \int_\Pi \psi\xi \quad
\mbox{ for all } i \geq 0.
\]
Since $\psi$ is bounded it follows that $f$ is finite-valued.
Since $\sW(\zeta_0,\leq i_1) \subset \sW(\zeta_0,\leq i_2)$ if $0 \leq i_1 \leq i_2$
it follows that $f$ is increasing.
The convexity of $\sW(\zeta_0,<\infty)$ ensures that
\[
(1-\theta)\sW(\zeta_0,\leq i_1) + \theta \sW(\zeta_0,\leq i_2)
\subset \sW(\zeta_0,\leq (1-\theta) i_1 + \theta i_2) \quad 
\forall i_1 \geq 0, \;  i_2 \geq 0, \; 0 \leq \theta \leq 1,
\]
hence $-f$ is a convex function.
Therefore $-f$ is continuous and subdifferentiable on $(0,\infty)$.

Now $\zeta$ maximises $\int_\Pi \psi\xi - f(I(\xi))$ subject to
$\xi \in \sW(\zeta_0,<\infty)$ so we have the subdifferential condition that,
for some $-\lambda \in \partial (-f)(i_0)$,
\begin{equation}
\zeta
\mbox{ maximises }
\int_\Pi \psi\xi -\lambda I(\xi) 
\left( = \int_\Pi (\psi -\lambda  x_2) \xi \right)
\mbox{ subject to }
\xi \in \sW(\zeta_0,<\infty) .
\label{eq6}
\end{equation}
Since $-f$ is decreasing we have $-\lambda \leq 0$ 
so $\psi-\lambda x_2$ is bounded above, from Lemma \ref{lm1}.

In order to derive a functional relationship between $\zeta$ and
$\Psi:=\psi-\lambda x_2$ from \eqref{eq6}, for each $n \in \NN$
we now write $Q(n)$ for the planar rectangle $(-n,n) \times (0,n)$
and consider the consequences of
rearranging $\zeta$ within $Q(n)$ while fixing $\zeta$ outside $Q(n)$.
Thus $\zeta 1_{Q(n)}$ maximises $\int_{Q(n)}\Psi \xi$ subject
to $\xi \in \sR_{Q(n)}(\zeta 1_{Q(n)})$ so by \cite[Lemma 2.15]{GRB:VARP} 
there exists an increasing function $\varphi_n$ such that
$\zeta =  \varphi_n \circ \Psi$ almost everywhere in $Q(n)$.
We can assume $\varphi_n$ to be defined on an interval $I_n$ such that
$\Psi(Q(n))^\circ \subset I_n \subset \overline{(\Psi(Q(n))}$.
We claim that all the $\varphi_n$ can be assumed to be restrictions of
a single increasing function $\varphi$ defined on the interval $I=\bigcup I_n$.
To see this, firstly define
$S_n=\{s \in I_n \mid \varphi_n(s) \neq \varphi_{n+1}(s) \}$.
Then $S_n$ must have empty interior relative to $I_n$, otherwise
$\varphi_n \circ \Psi \neq \varphi_{n+1} \circ \Psi$ throughout a nonempty
open subset of $Q(n)$, which must have positive measure.
Hence $\varphi_n =\varphi_{n+1}$ on a dense subset of $I_n$ and therefore
at all points where $\varphi_n$ and $\varphi_{n+1}$ are both continuous
relative to $I_n$.
Let $D \subset I$ comprise all discontinuities of the $\varphi_n$, $n \in \NN$,
which is a countable set by monotonicity; then $\varphi_k(s)=\varphi_n(s)$
provided that $s \in I_n \setminus D$ and $k>n$.
Let $T$ be the set of $s \in I$ for which $\Psi^{-1}(s)$ has positive measure,
which is also countable.
If $s \in T$ and $n$ is the least number for which $\Psi^{-1}(s) \cap Q(n)$
has positive measure then we must have $\varphi_k(s)=\varphi_n(s)$ for all
$k>n$, because $\varphi_k \circ \Psi = \varphi_n \circ \Psi$ almost
everywhere on $Q(n)$; hence $\varphi_k(\Psi(x))=\varphi_n(s)$ for almost
every $x \in Q(k) \cap \Psi^{-1}(s)$, for every $k \in \NN$.
For $s \in I \setminus (D \setminus T)$ we can now define
$\varphi(s)=\varphi_n(s)$ for all sufficiently large $n$ and find that
$\varphi$ is increasing and $\zeta=\varphi \circ \Psi$
almost everywhere on $\Pi \setminus \Psi^{-1}(D \setminus T)$.
Since $\Psi^{-1}(D \setminus T)$ has zero measure,
we can complete the construction of $\varphi$ by adopting any definition of
$\varphi(s)$ for $s \in D \setminus T$ that makes $\varphi$ increasing on $I$.

We defer the proof that $\lambda$ is strictly positive until after (iv).

For (iv), let $\zeta$, $\psi$, $\varphi$ and $\lambda \geq 0$ be as above.
Since $\psi$ is Steiner symmetric about the $x_2$-axis, Lemma \ref{lmx2}
yields a constant $c_6>0$ such that
\begin{equation}
\psi(x) \leq
c_6(I(\zeta)+\|\zeta\|_1+\|\zeta\|_p) x_2 \min\{1,| x_1 |^{-1/(2p)}\},
\quad x \in \Pi .
\label{eq2}
\end{equation}
Since $\zeta$ maximises $\int_\Pi (\psi-\lambda x_2) \xi$ subject to
$\xi \in \sW(\zeta_0,<\infty)$ we must have $\zeta=0$ almost everywhere
in the set 
$A = \{ x \in \Pi \mid \psi(x)-\lambda x_2 <0\}$, otherwise
\[
\int (\psi-\lambda x_2) \zeta <
\int (\psi-\lambda x_2) \zeta 1_{\Pi \setminus A}
\]
which is impossible since $\zeta 1_{\Pi \setminus A} \preceq \zeta$ so
$\zeta 1_{\Pi \setminus A} \in \sW(\zeta_0,<\infty)$.
The set of $x$ where $\zeta(x)(=-\Delta\psi(x))$ is positive and
$\psi(x)-\lambda x_2=0$ necessarily has zero measure.

If $\lambda>0$ it now follows from \eqref{eq2}  that $\zeta$ vanishes almost
everywhere outside the region where
$x_1^{1/(2p)} \leq c(I(\zeta) + \|\zeta\|_1 + \|\zeta\|_p)/\lambda$,
so the support of $\zeta$ is also bounded in the $x_1$ direction.

If $\lambda=0$ then $\zeta=\varphi \circ \psi$ almost everywhere in $\Pi$ so,
since $\varphi$ is increasing, there exists $\kappa \geq 0$ such that
\[
\{ x \in \Pi \mid \psi(x) > \kappa \}
\subset \{ x \in \Pi \mid \zeta(x) >0 \}
\subset \{ x \in \Pi \mid \psi(x) \geq \kappa \}
\]
apart from sets of measure zero.
Thus $\kappa>0$, for otherwise $\zeta>0$ almost everywhere on $\Pi$,
whereas $\zeta_0$ vanishes outside a set of finite measure
and $\zeta \in \sR\sC(\zeta_0)$.
It then follows from \eqref{eq2} that the support of $\zeta$ lies within
the region
$\kappa x_1^{1/(2p)} \leq c_6(I(\zeta)+\|\zeta\|_1+\|\zeta\|_p) Z$.
Thus $\zeta$ vanishes outside a bounded region in the case $\lambda=0$ also.

We now return to (iii) and let $\zeta$, $\varphi$ and $\lambda \geq 0$ be as
above.
We can assume $\varphi$ to be non-negative throughout $\RR$ and to vanish on
$(-\infty,0]$, we define $\Phi(s)=\int_{-\infty}^s \varphi$ and we note that
$\Phi(s)>0$ if and only if $\varphi(s)>0$, because $\varphi$ is increasing.
From \cite[Lemma 9]{GRB:VP} we have
\[
2 \int_Q \Phi(\sG\zeta(x) - \lambda x_2) \rd x
- \lambda \int_Q \zeta(x)x_2 \rd x
= \int_{\partial Q} \Phi(\sG\zeta(x) - \lambda x_2)(x \cdot \mathfrak{n}) \rd x
\]
where $Q=[-R,R] \times [0,R]$ is a rectangle containing the support
of $\zeta$ and $\mathfrak{n}$ is the outward unit normal.
Since $\varphi\circ(\sG\zeta-\lambda x_2)$ vanishes outside $Q$, so too
does $\Phi\circ(\sG\zeta-\lambda x_2)$ and we deduce that
\begin{equation}
2 \int_\Pi \Phi(\sG\zeta(x) - \lambda x_2) \rd x = \lambda I(\zeta).
\label{eqx9}
\end{equation}
Since $\Phi(\sG\zeta(x) - \lambda x_2)>0$ almost always when $\zeta(x)>0$
we deduce from \eqref{eqx9} that $\lambda>0$.
\end{proof}

The next result shows that the hypotheses of Theorems \ref{thm1} and
\ref{thm2} below are satisfied in a wide range of situations.

\begin{lemma}
Let $2<p<\infty$, let $\zeta_0 \in L^1(\Pi) \cap L^p(\Pi)$ be non-negative
and have compact support, let $i>0$ and let $\Sigma_0$ be the set of
maximisers of $E$ relative to $\sW(\zeta_0,\leq i)$.
Then for all sufficiently large $i$ we have $\Sigma_0  \subset \sR(\zeta_0)$.
\label{lmx7}
\end{lemma}

\begin{proof}
We use the arguments from the corresponding part of the proof of
\cite[Theorem 16(ii)]{GRB:VP}.
Consider $i>0$, let $\Sigma_i$ denote the set of maximisers of $E$
relative to $\sW(\zeta_0, \leq i)$ and let $M_i$ denote the maximum value.
Consider $\zeta \in \Sigma_i$. 

Then, from Lemma \ref{lmx6}, $\zeta \in \sR\sC(\zeta_0)$ and there exist
$\lambda >0$ and an increasing function $\varphi$ such that
$\zeta(x) = \varphi (\sG\zeta(x) - \lambda x_2)$
except for a set of $x$ having measure zero.
There must be a number $\beta$ such that 
$\varphi(s)>0$ for $s>\beta$ and $\varphi(s)=0$ for $s<\beta$.
Moreover $\beta \geq 0$, for if $\beta<0$ then
$\sG\zeta(x)-\lambda x_2>\beta$ for almost all $x$ satisfying
$0<x_2<-\beta/\lambda$ whereas $\zeta$ vanishes outside a set of finite measure.

Now, from the definitions,
\[
\int \zeta(x) (\sG\zeta(x)-\lambda x_2) \rd x
= 2E(\zeta)-\lambda I(\zeta)
\]
and, by \cite[Lemma 9]{GRB:VP},
\[
2E(\zeta) = \int \zeta \sG\zeta
\geq \frac{3}{2} \lambda I(\zeta) + \beta \|\zeta\|_1
\geq \frac{3}{2} \lambda I(\zeta)
\]
so we obtain
\[
\int \zeta(x) (\sG\zeta(x)-\lambda x_2) \rd x
\geq \frac{2}{3}E(\zeta) = \frac{2}{3}M_i .
\]
Since $\|\zeta\|_1 \leq \|\zeta_0\|_1$ we deduce
\[
S := \sup \{ \sG\zeta(x) -\lambda x_2 \mid x \in \Pi \}
\geq \frac{2 M_i}{3\|\zeta_0\|_1} .
\]
Let $z$ be a point where $\sG\zeta(x)-\lambda x_2$ achieves its supremum.
Then, for $x$ with $x_2<z_2$ we may apply the mean value inequality
along the line segment $[z,x]$ and use Lemma \ref{lmx1} to obtain
\[
\sG\zeta(x) - \lambda x_2
\geq \sG\zeta(x) - \lambda z_2
\geq \sG\zeta(z) - \lambda z_2 - c_5(\|\zeta_0\|_1+\|\zeta_0\|_p)|x-z|
\]
and this is positive provided $|x-z|<S/(c_5(\|\zeta_0\|_1+\|\zeta_0\|_p))$,
for which it is sufficient that
$|z-x| \leq 2M_i/(3c_5\|\zeta_0\|_1(\|\zeta_0\|_1+\|\zeta_0\|_p))$.
Since $M_i \to \infty$ as $i \to \infty$
it follows that we can choose $i_1>0$ such that if $i>i_1$ then the area of
the set
$\{x \mid \sG\zeta(x)-\lambda x_2>0\}$ is greater than the area of the
set $\{ x \mid \zeta(x)>0 \}$.

If $i>i_1$ we claim that $\zeta \in \sR(\zeta_0)$.
Suppose not; then the supports of $\zeta^\Delta$ and $\zeta_0^\Delta$
would be intervals $[0,s]$ and $[0,t]$ respectively with $s<t$.
Then, for some $s<r<t$ the isomorphism construction described in
Section \ref{rearr} would yield a rearrangement $\eta$ of
$\zeta_0^\Delta 1_{[s,r]}$ on a subset of
$\{ x \mid \sG\zeta(x)-\lambda x_2>0, \, \zeta(x)=0\}$
and then we would have $\xi:=\zeta+\eta \in \sR\sC(\zeta_0)$ and
\[
\int \xi(x)(\sG\zeta(x)-\lambda x_2)
> \int \zeta(x)(\sG\zeta(x)-\lambda x_2) .
\]
From this it would follow that $E(\xi)>E(\zeta)$.
This would be impossible, so $\zeta \in \sR(\zeta_0)$ as claimed.
\end{proof}

\section{Proofs of the Compactness and Stability Theorems}
\label{S4}

\subsection{Proof of Theorem \ref{thm1}}
\label{proof1}

We note that $\Sigma_0$ is equal to the set of maximisers of $E$ relative
to $\sW(\zeta_0,\leq i_0)$, from Lemma \ref{lmx6}(i).
We write
\[
\beta = \int_\Pi \zeta_0 = \lim_{n \to \infty} \int_\Pi \zeta_n .
\]
By concentration-compactness \cite[Lemma I.1]{PLL:CC1}, we can replace
$(\zeta_n)$ by a subsequence having one of the following properties:\\
{\em Dichotomy:}
For each $n \in \NN$ there is a partition of $\Pi$ into measurable sets
$\Omega_n^1$, $\Omega_n^2$ and $\Omega_n^3$ in such a way that
$\zeta_n^k = \zeta_n 1_{\Omega_n^k}$, $k=1,2,3$, satisfy
\begin{eqnarray*}
\int_\Pi \zeta_n^1 \to \alpha,\\
\int_\Pi \zeta_n^2 \to \beta-\alpha,\\
\int_\Pi \zeta_n^3 \to 0,\\
\mbox{dist}(\Omega_n^1,\Omega_n^2) \to 0
\end{eqnarray*}
as $n \to \infty$, where $0< \alpha < \beta$.

\noindent
{\em Vanishing:}
\[
\forall R>0 \quad \lim_{n\to\infty} \sup_{y \in \Pi} \int_{D_\Pi(y,R)} \zeta_n
= 0;
\]

\noindent
{\em Compactness:}
there exists a sequence $(y^n)$ in $\overline{\Pi}$ such that
\[
\forall \varepsilon>0 \exists R>0 \mbox{ s.t. } \forall n \in \NN
\int_{D_\Pi(y^n,R)} \zeta_n > \beta - \varepsilon .
\]
We show that Dichotomy and Vanishing cannot occur and deduce the result
from Compactness.

\bigskip

\noindent
{\bf Excluding Dichotomy.}\\
We have 
\[
\int_\Pi \zeta_n^3 (\sG \zeta_n) \to 0 \mbox{ as } n \to \infty
\]
because $\sG \zeta_n$ is uniformly bounded by Lemmas \ref{lm1} and \ref{lmx1}
and $\| \zeta_n^3 \|_1 \to 0$.
Now
\[
E(\zeta_n) \geq E(\zeta_n^1+\zeta_n^2) = E(\zeta_n -\zeta_n^3)
 = E(\zeta_n) + E(\zeta_n^3) - \int_\Pi (\sG\zeta_n)\zeta_n^3
\geq E(\zeta_n) - \int_\Pi (\sG\zeta_n)\zeta_n^3 \to M_0
\]
and since $E(\zeta_n) \to M_0$ it follows that
\begin{equation}
E(\zeta_n^1 + \zeta_n^2) \to M_0.
\label{eqx10}
\end{equation}
We claim
\begin{equation}
\int_\Pi \zeta_n^1 \sG \zeta_n^2 \to 0.
\label{eqx1}
\end{equation}
To prove \eqref{eqx1} note firstly that, given $\varepsilon>0$, we can by 
Lemma \ref{lm1} choose $W>0$ independent of $n$ such that
$\sG\zeta_n^1(x) < \varepsilon$ and $\sG\zeta_n^2(x) < \varepsilon$
if $x_2>W$, thus
\begin{eqnarray*}
\int_{y_2>W} \zeta_n^2(y) \sG\zeta_n^1(y) \rd y 
&\leq& \varepsilon \| \zeta_n \|_1 , \\
\int_{x_2>W} \zeta_n^1(x) \sG\zeta_n^2(x)\rd x
&\leq& \varepsilon \| \zeta_n \|_1  .
\end{eqnarray*}
The remaining term in \eqref{eqx1} is
\[
\int_{x_2<W} \int_{y_2<W} G(x,y)\zeta_n^1(x) \zeta_n^2(y) \rd x\rd y
\leq \frac{W^2}{\pi\,\mathrm{dist}(\Omega_n^1,\Omega_n^2)^2} \|\zeta_n\|_1^2
\to 0 \mbox{ as } n \to \infty.
\]
Hence \eqref{eqx1}, which, together with \eqref{eqx10}, shows that 
\begin{equation*}
E(\zeta_n^1) + E(\zeta_n^2) \to M_0.
\end{equation*}

For $k=1,2$ let $\zeta_n^{k*}$ denote the Steiner symmetrisation of
$\zeta_n^k$ about the $x_2$-axis, so $E(\zeta_n^{k*}) \geq E(\zeta_n^k)$
and $I(\zeta_n^{k*}) = I(\zeta_n^k) = i_n^k$ say. 
We can pass to a subsequence and suppose that $\zeta_n^{k*} \to \zeta^k$ say,
weakly in $L^p$.
Now $E$ is continuous with respect to $L^p$ weak convergence of
Steiner-symmetric sequences when $\| \, \|_1$ and $I$ are bounded, by
Lemma \ref{lmx4}, so $E(\zeta^1)+E(\zeta^2) \geq M_0$, whereas
$I(\zeta^1)+I(\zeta^2) \leq i_0$ by weak lower semicontinuity of $I$
relative to non-negative functions in $L^p$.
The decreasing rearrangements $\zeta_n^{1\Delta}$ and $\zeta_n^{2\Delta}$
are both dominated by $\zeta_n^\Delta$ which converges in $L^1$ to
$\zeta_0^\Delta$, so a variant of
Helly's Selection Principle for monotonic functions (see \cite{KOLFOM})
shows that we can pass to a subsequence and suppose $\zeta_n^{1\Delta}$
and $\zeta_n^{2\nabla}$ converge pointwise, and then deduce that they
converge strongly in
$L^1$, to non-negative functions $\xi^1$ and $\xi^2$, say, dominated by
$\zeta_0^\Delta$ and $\zeta_0^\nabla$ respectively, thus the $\xi^k$
are supported on bounded intervals.

Since $\zeta_n^{k*}$ is a rearrangement of $\zeta_n^{k\Delta}$ we have
$\zeta_n^{k*} \preceq \zeta_n^{k\Delta}$ and since
the right-hand integral in \eqref{eqx8} is strongly continuous in
$L^1$ whereas the left-hand integral is weakly lower semicontinuous
in $L^p$, we deduce $\zeta^k \preceq \xi^k$.
On the other hand $\zeta_n^1 + \zeta_n^2 \preceq \zeta_0$ so
$\zeta_n^{1\Delta} + \zeta_n^{2\nabla} \preceq \zeta_0$, since the left-hand
integral of \eqref{eqx8} is additive over two functions that are simultaneously
positive almost nowhere, thus once more we can pass to the limit in
\eqref{eqx8} to obtain $\xi^1 + \xi^2 \preceq \zeta_0$.

Let $i^k = I(\zeta^k)$ for $k=1,2$ so $i^1+i^2 \leq i_0$ and let
$\widetilde{\xi}^k$ be a maximiser for $E$ relative to $\sW(\xi^k, \leq i^k)$,
so that $E(\widetilde{\xi}^k) \geq E(\zeta^k)$. 
Lemma \ref{lmx6} shows that the $\widetilde{\xi}^k$ exist and have
compact supports, say in a common rectangle $[-Q,Q] \times [0,Q]$.
Then define $\widehat{\xi}^1(x_1,x_2):=\widetilde{\xi}^1(x_1+Q,x_2)$ and
$\widehat{\xi}^2(x_1,x_2):=\widetilde{\xi}^2(x_1-Q,x_2)$, which are
simultaneously positive almost nowhere and satisfy
$\widehat{\xi}^k \preceq \xi^k$.
Again the additivity of the integrals in \eqref{eqx8} ensures that
$\widehat{\xi}^1 + \widehat{\xi}^2 \preceq \xi^1 + \xi^2 \preceq \zeta_0$ so
$\widehat{\xi}^1 + \widehat{\xi}^2 \in \sW(\zeta_0,\leq i_0)$.
Now
\begin{equation}
E(\widehat{\xi}^1 + \widehat{\xi}^2)
= E(\widehat{\xi}^1) + E(\widehat{\xi}^2)+ \int_\Pi \widehat{\xi}^1 \sG \widehat{\xi}^2
\geq E(\widehat{\xi}^1) + E(\widehat{\xi}^2) \geq M_0
\label{eqx5}
\end{equation}
proving that $\widehat{\xi}^1 + \widehat{\xi}^2$ is a maximiser for $E$ relative to
$\sW(\zeta_0, \leq i_0)$.
Since $\Sigma_0 \subset \sR(\zeta_0)$ by hypothesis, we now have
$\widehat{\xi}^1 + \widehat{\xi}^2 \in \sR(\zeta_0)$.

We have
\begin{eqnarray*}
\int_\Pi \widehat{\xi}^1 \leq \int_\Pi \xi^1
&\leq& \lim_{n\to\infty} \int_\Pi \zeta_n^1 = \alpha , \\
\int_\Pi \widehat{\xi}^2 \leq \int_\Pi \xi^2
&\leq& \lim_{n\to\infty} \int_\Pi \zeta_n^2 = \beta-\alpha .
\end{eqnarray*}
If $\widehat{\xi}^1  = 0$ or $\widehat{\xi}^2  = 0$ then 
\[
\int_\Pi \widehat{\xi}^1 + \widehat{\xi}^2 < \alpha + (\beta-\alpha) = \beta
= \int_\Pi \zeta_0
\]
contradicting $\widehat{\xi}^1 + \widehat{\xi}^2 \in \sR(\zeta_0)$.
Therefore $\widehat{\xi}^1$ and $\widehat{\xi}^2$ are both nonzero so
the first inequality of \eqref{eqx5} is strict, which is impossible.
Thus Dichotomy does not occur.

\bigskip

\noindent
{\bf Excluding Vanishing.}\\
Let $\varepsilon>0$ and choose by Lemma \ref{lm1} $W>0$ large enough that
$\sG\zeta_n(x_1,x_2)<\varepsilon$ for all $n$ if $x_2>W$.
We write
\begin{eqnarray*}
\Pi_W &=& \{ (x_1,x_2) \mid 0 < x_2 < W \} \\
\Pi^W &=& \{ (x_1,x_2) \mid  x_2 > W \} 
\end{eqnarray*}
and deduce that
\[
\int_\Pi \int_{\Pi^W} G(x,y) \zeta_n(x) \zeta_n(y) \rd x\rd y
\leq \varepsilon \| \zeta_n \|_1 .
\]
Then for $x \in \Pi_W$ and $R>0$, writing $\rho=|x-y|$,
\begin{eqnarray*}
\int_{\Pi_W \setminus D(x,R)} G(x,y)\zeta_n(y) \rd y
&\leq& 
\int_{\Pi_W \setminus D(x,R)} \frac{1}{4\pi}
\log \left( 1+ \frac{4x_2 y_2}{\rho^2} \right) \zeta_n(y) \rd y\\
&\leq&
\int_{\Pi_W \setminus D(x,R)} \frac{1}{4\pi}
\log \left( 1+ \frac{4x_2 (x_2+\rho)}{\rho^2} \right) \zeta_n(y) \rd y\\
&\leq&
\frac{1}{4\pi} \log \left( 1+ \frac{4W (W+R)}{R^2} \right) \| \zeta_n \|_1
\leq \varepsilon \| \zeta_n \|_1
\end{eqnarray*}
provided we choose $R$ suitably large, independently of $n$.
Again for $x \in \Pi_W$ and $R>0$ chosen as above we have
\begin{eqnarray*}
\int_{\Pi_W \cap D(x,R)} G(x,y)\zeta_n(y) \rd y
&\leq& 
\int_{\Pi_W \cap D(x,R)} 
\frac{1}{4\pi} \log \left( 1+ \frac{4W (W+R)}{\rho^2} \right) \zeta_n(y)
\rd y \\
&\leq&
\left( \int_{D(x,R)}
\left( \log \left( 1+ \frac{4W (W+R)}{\rho^2} \right) \right)^r \right)^{1/r}
\| \zeta_n \|_{L^1(D_\Pi(x,R))}^\theta \| \zeta_n \|_p^{1-\theta} ,
\end{eqnarray*}
where $1/r+1/s=1$, $1<s<p$ and $\theta + (1-\theta)/p = 1/s$, by H\"{o}lder's
inequality and the interpolation inequality.
Since $\| \zeta_n \|_{L^1(D_\Pi(x,R))} \to 0$ as $n \to \infty$ uniformly over
$x \in \Pi$ by assumption of Vanishing and $\| \zeta_n \|_p$ is bounded
we now have
\[
\int_{\Pi_W \cap D(x,R)} G(x,y)\zeta_n(y) \rd y
< \varepsilon
\]
for all sufficiently large $n$, uniformly over $x \in \Pi_W$.

Thus
\[
\int_{\Pi_W} G(x,y) \zeta_n(y) \rd y
< \varepsilon \| \zeta_n \|_1 + \varepsilon
\quad \mbox{ for all } x \in \Pi_W
\]
for all sufficiently large $n$ and therefore
\[
\int_{\Pi_W} \int_{\Pi_W} G(x,y)\zeta_n(y)\zeta_n(x)\rd y\rd x
< \varepsilon (\|\zeta_n\|_1^2 + \| \zeta_n \|_1)
\]
for all sufficiently large $n$.
Now
\begin{eqnarray*}
E(\zeta_n) &=& \left( \frac{1}{2}\int_{\Pi_W}\int_{\Pi_W}
+ \int_{\Pi_W}\int_{\Pi^W} + \frac{1}{2}\int_{\Pi^W}\int_{\Pi^W} \right)
G(x,y) \zeta_n(x) \zeta_n(y) \rd x\rd y\\
&\leq& \left( \frac{1}{2}\int_{\Pi_W}\int_{\Pi_W}
+ \int_\Pi\int_{\Pi^W} \right) G(x,y) \zeta_n(x) \zeta_n(y) \rd x\rd y
\leq  \varepsilon \left(\frac{1}{2} \| \zeta_n \|_1^2 + \frac{1}{2}  \| \zeta_n \|_1
+ \| \zeta_n \|_1 \right)
\end{eqnarray*}
for all sufficiently large $n$, hence $E(\zeta_n) \to 0$ as $n \to \infty$.
Since $M_0$ is positive, Vanishing cannot occur for a maximising sequence.

\bigskip

\noindent
{\bf Exploiting Compactness.}\\
If $y^n_2 \to \infty$ as $n \to \infty$ then, for each fixed $R>0$,
we would have
\[
\int_{D_\Pi(y^n,R)} \zeta_n \leq (y^n_2-R)^{-1}I(\zeta_n) \to 0
\]
as $n \to \infty$, contradicting the assumption of Compactness.
Therefore, after passing to a further subsequence if necessary, we can suppose
that $y^n_2 <W$ for all $n$, where $W>0$ is fixed.

The Compactness assumption ensures that
\begin{equation}
\sup_{n \in \NN} \int_{\Pi \setminus D(y^n,R)} \zeta_n \to 0 
\mbox{ as } R \to \infty .
\label{eq10}
\end{equation}
We have $D((y_1^n,y_2^n),R) \subset D((y_1^n,0),R+W)$ so there is no loss
in supposing that $y_2^n=0$ for all $n$.
Since $\sG \zeta_n$ is bounded in $L^\infty$ uniformly over $n$ we deduce
\begin{equation*}
\sup_{n \in \NN}
\int_{\Pi \setminus D(y^n,R)} \zeta_n\sG\zeta_n
\to 0 \mbox{ as } R \to \infty .
\end{equation*}
We write $\overline{\zeta}_n (x_1,x_2) := \zeta_n (x_1 - y^n_1,x_2)$.
It follows that
\begin{equation}
\sup_{n\in\NN} \int_{\Pi \setminus D(0,R)} \int_\Pi
G(x,y) \overline{\zeta}_n(x) \overline{\zeta}_n(y)\rd x \rd y
\to 0 \mbox{ as } R \to \infty .
\label{eq11}
\end{equation}

Now, for fixed $R>0$, $\sG$ followed by restriction to $D_\Pi(0,R)$ acts as a
compact operator from $L^p(D_\Pi(0,R))$ to $L^q(D_\Pi(0,R))$, where
$p^{-1}+q^{-1}=1$, and we can further pass to a subsequence to ensure
$\overline{\zeta}_n \to \overline{\zeta} \in \sW(\zeta_0,\leq i_0)$ say,
weakly in $L^p(\Pi)$.
Then, for each fixed $R>0$,
\begin{equation}
\int_{D_\Pi(0,R)} \int_{D_\Pi(0,R)} G(x,y)
\overline{\zeta}_n(x) \overline{\zeta}_n(y) \rd x\rd y  \to
\int_{D_\Pi(0,R)} \int_{D_\Pi(0,R)} G(x,y)
\overline{\zeta}(x) \overline{\zeta}(y) \rd x\rd y \mbox{ as } n \to \infty .
\label{eq12}
\end{equation}
From \eqref{eq11} and \eqref{eq12} we deduce
\[
E(\overline{\zeta}_n) \to E(\overline{\zeta}).
\]
Thus $\overline{\zeta}$ is a maximiser for $E$ relative to
$\sW(\zeta_0,\leq i_0)$.
Therefore $\overline{\zeta} \in \sR(\zeta_0)$ so
\[
\|\overline{\zeta}\|_p=\|\zeta_0\|_p=\lim_{n\to\infty}\|\overline{\zeta}_n\|_p
\]
hence by uniform convexity $\overline{\zeta}_n \to \overline{\zeta}$
strongly in $L^p(\Pi)$.

It follows that $\overline{\zeta}_n \to \overline{\zeta}$ strongly in
$L^1(D_\Pi(0,R))$ for each $R>0$.
In view of \eqref{eq10} it now follows that
$\overline{\zeta}_n \to \overline{\zeta}$
strongly in $L^1(\Pi)$.
Since $\overline{\zeta}$ is a maximiser we have $I(\overline{\zeta})=i_0$
so $I(\overline{\zeta}_n) \to I(\overline{\zeta})$.
\qedsymbol

\subsection{Proof of Theorem \ref{thm2}}
\label{proof2}
As previously stated, we view vorticity $\omega$ as a function of time $t$
taking values in $L^1(\Pi) \cap L^p(\Pi)$ and suppress the space variable $x$.

Suppose the result fails.
Then there exists $\varepsilon>0$ such that, for all sufficiently large
$n \in \NN$, we can choose a solution $\omega_n(\cdot)$ of the vorticity
equation and a time $t_n$, such that
$\mathrm{dist}_\Xp (\omega_n(0),\Sigma_0)<1/n$
but 
$\mathrm{dist}_\Xp (\omega_n(t_n),\Sigma_0) \geq \varepsilon$.

Observe that, by the conservation properties of the vorticity equation
\[
I(\omega_n(t_n))=I(\omega_n(0)) \to i_0,
\]
\[
\mathrm{dist}_1(\omega_n(t_n),\sR(\zeta_0))
= \mathrm{dist}_1(\omega_n(0),\sR(\zeta_0))
\leq \mathrm{dist}_1(\omega_n(0),\Sigma_0) \to 0,
\]
\[
\mathrm{dist}_p(\omega_n(t_n),\sR(\zeta_0))
= \mathrm{dist}_p(\omega_n(0),\sR(\zeta_0))
\leq \mathrm{dist}_p(\omega_n(0),\Sigma_0) \to 0
\]
as $n \to \infty$, hence using inequality \eqref{eq7},
\begin{eqnarray*}
\| \omega_n(t_n)^\Delta -\zeta_0^\Delta \|_p &\to& 0, \\
\| \omega_n(t_n)^\Delta -\zeta_0^\Delta \|_1 &\to& 0 .
\end{eqnarray*}
Moreover, using Lemma \ref{lmx5} in addition,
\[
E(\omega_n(t_n))=E(\omega_n(0)) \to M_0
\]
as $n \to \infty$.
Theorem \ref{thm1} now ensures that, after passing to a
subsequence, we can choose an $x_1$ translation $\xi_n$ of each
$\omega_n(t_n)$ such that the sequence $(\xi_n)$ converges to
an element $\xi_0$ of $\Sigma_0$ in $\| \,\|_1 + \| \,\|_p$.
Since $x_1$ translations preserve $\Sigma_0$ and $\| \, \|_\Xp$ we have
\[
\mathrm{dist}_\Xp(\omega_n(t_n),\Sigma_0)=
\mathrm{dist}_\Xp(\xi_n,\Sigma_0)
\leq \| \xi_n -\xi_0 \|_\Xp \to 0
\]
and this contradicts the choice of the $\omega_n$ and $t_n$, completing the
proof.
\qedsymbol

\bigskip

\noindent
{\bf Examples.}

Given arbitrary compactly supported nontrivial non-negative $\zeta_0$
in $L^p(\Pi)$ for some finite $p>2$, the hypotheses of Theorem 2 are
satisfied for all sufficiently large $i_0$, by Lemma \ref{lmx7}.

\bigskip
\noindent
{\bf\large Acknowledgement.}

The author thanks Helena and Milton Lopes for valuable conversations and he
is grateful for the hospitality and support of
the {\em Thematic Programme on Incompressible Fluid Dynamics}
held in 2014 at IMPA, Rio de Janeiro, 
where this research was substantially conducted.

\end{document}